\documentclass[11pt, b5paper]{article}
\usepackage{authblk}

\makeatletter
\def\blfootnote{\gdef\@thefnmark{}\@footnotetext}
\makeatother

\usepackage{tikz, mathrsfs}

\usetikzlibrary{arrows,cd}

\tikzset{>=stealth'}

\def\arrowLengthDisplayStyle{4ex}
\def\arrowHeightDisplayStyle{.8ex}
\def\arrowSkipDisplayStyle{.5ex}
\def\arrowLengthTextStyle{3ex}
\def\arrowHeightTextStyle{.8ex}
\def\arrowSkipTextStyle{.4ex}
\def\arrowLengthScriptStyle{2.5ex}
\def\arrowHeightScriptStyle{.6ex}
\def\arrowSkipScriptStyle{.3ex}
\def\arrowLengthScriptScriptStyle{2ex}
\def\arrowHeightScriptScriptStyle{.4ex}
\def\arrowSkipScriptScriptStyle{.2ex}

\renewcommand{\to}{\arrow{->}}

\newcommand{\MakeTikzArrowWithSuperscriptSubscript}[4]
{
	\mathchoice
	{ 
		\hspace*{\arrowSkipDisplayStyle}
		\begin{tikzpicture}[baseline]
		\draw [#1] (0,\arrowHeightDisplayStyle) -- node [above] {$#2$} node [below] {$#3$} (#4 * \arrowLengthDisplayStyle, \arrowHeightDisplayStyle);
		\end{tikzpicture}
		\hspace*{\arrowSkipDisplayStyle}
	}
	{ 
		\hspace*{\arrowSkipTextStyle}
		\begin{tikzpicture}[baseline]
		\draw [#1] (0,\arrowHeightTextStyle) -- node [above] {$\scriptstyle #2$} node [below] {$\scriptstyle #3$} (#4 * \arrowLengthTextStyle, \arrowHeightTextStyle);
		\end{tikzpicture}
		\hspace*{\arrowSkipTextStyle}
	}
	{ 
		\hspace*{\arrowSkipScriptStyle}
		\begin{tikzpicture}[baseline]
		\draw [#1] (0,\arrowHeightScriptStyle) -- node [above] {$\scriptscriptstyle #2$} node [below] {$\scriptscriptstyle #3$} (#4 * \arrowLengthScriptStyle, \arrowHeightScriptStyle);
		\end{tikzpicture}
		\hspace*{\arrowSkipScriptStyle}
	}
	{ 
		\hspace*{\arrowSkipScriptScriptStyle}
		\begin{tikzpicture}[baseline]
		\draw [#1] (0,\arrowHeightScriptScriptStyle) -- node [above] {$\scriptscriptstyle #2$} node [below] {$\scriptscriptstyle #3$} (#4 * \arrowLengthScriptScriptStyle, \arrowHeightScriptScriptStyle);
		\end{tikzpicture}
		\hspace*{\arrowSkipScriptScriptStyle}
	}
}

\newcommand{\MakeTikzArrowWithCentralLabel}[3]
{
	\mathchoice
	{ 
		\hspace*{\arrowSkipDisplayStyle}
		\begin{tikzpicture}[baseline]
		\draw [#1] (0,\arrowHeightDisplayStyle) -- node [fill=white,inner sep=1pt] {$#2$} (#3 * \arrowLengthDisplayStyle, \arrowHeightDisplayStyle);
		\end{tikzpicture}
		\hspace*{\arrowSkipDisplayStyle}
	}
	{ 
		\hspace*{\arrowSkipTextStyle}
		\begin{tikzpicture}[baseline]
		\draw [#1] (0,\arrowHeightTextStyle) -- node [fill=white,inner sep=1pt] {$\scriptstyle #2$} (#3 * \arrowLengthTextStyle, \arrowHeightTextStyle);
		\end{tikzpicture}
		\hspace*{\arrowSkipTextStyle}
	}
	{ 
		\hspace*{\arrowSkipScriptStyle}
		\begin{tikzpicture}[baseline]
		\draw [#1] (0,\arrowHeightScriptStyle) -- node [fill=white,inner sep=1pt] {$\scriptscriptstyle #2$} (#3 * \arrowLengthScriptStyle, \arrowHeightScriptStyle);
		\end{tikzpicture}
		\hspace*{\arrowSkipScriptStyle}
	}
	{ 
		\hspace*{\arrowSkipScriptScriptStyle}
		\begin{tikzpicture}[baseline]
		\draw [#1] (0,\arrowHeightScriptScriptStyle) -- node [fill=white,inner sep=1pt] {$\scriptscriptstyle #2$} (#3 * \arrowLengthScriptScriptStyle, \arrowHeightScriptScriptStyle);
		\end{tikzpicture}
		\hspace*{\arrowSkipScriptScriptStyle}
	}
}

\def\arrow#1{\def\lastArrowStyle{#1}
	\futurelet\testchar\arrowMaybeStreched}
\def\arrowMaybeStreched{\ifx[\testchar \let\next\arrowStreched
	\else \let\next\arrowUnstreched \fi
	\next}

\def\arrowStreched[#1]{\def\lastArrowStrech{#1}
	\futurelet\testchar\arrowMaybeLabel}
\def\arrowUnstreched{\def\lastArrowStrech{1}
	\futurelet\testchar\arrowMaybeLabel}

\def\arrowMaybeLabel{\ifx^\testchar \let\next\arrowSuperscript
	\else \ifx_\testchar \let\next\arrowSubscript
	\else \ifx~\testchar \let\next\arrowCentralLabel
	\else \let\next\arrowNoLabel
	\fi
	\fi
	\fi
	\next}

\def\arrowSuperscript^#1{\def\lastArrowSuperscript{#1}
	\futurelet\testchar\arrowSuperMaybeSub}
\def\arrowSuperMaybeSub{\ifx_\testchar \let\next\arrowSuperscriptSubscript
	\else \let\next\arrowSuperscriptNoSubscript \fi
	\next}

\def\arrowSubscript_#1{\def\lastArrowSubscript{#1}
	\futurelet\testchar\arrowSubMaybeSuper}
\def\arrowSubMaybeSuper{\ifx^\testchar \let\next\arrowSubscriptSuperscript
	\else \let\next\arrowSubscriptNoSuperscript \fi
	\next}

\def\arrowSuperscriptSubscript_#1{\def\lastArrowSubscript{#1}
	\arrowDrawSupSub}
\def\arrowSuperscriptNoSubscript{\def\lastArrowSubscript{}
	\arrowDrawSupSub}
\def\arrowSubscriptSuperscript^#1{\def\lastArrowSuperscript{#1}
	\arrowDrawSupSub}
\def\arrowSubscriptNoSuperscript{\def\lastArrowSuperscript{}
	\arrowDrawSupSub}
\def\arrowNoLabel{\def\lastArrowSuperscript{}
	\def\lastArrowSubscript{}
	\arrowDrawSupSub}

\def\arrowCentralLabel~#1{\MakeTikzArrowWithCentralLabel{\lastArrowStyle}{#1}{\lastArrowStrech}}
\def\arrowDrawSupSub{\MakeTikzArrowWithSuperscriptSubscript{\lastArrowStyle}{\lastArrowSuperscript}{\lastArrowSubscript}{\lastArrowStrech}}

\tikzcdset{arrow style=tikz, diagrams={>=stealth'}, row sep=4em, column sep=4em}

\usepackage{xfrac} 
\usepackage{amsmath,amssymb,amsthm}
\usepackage{mathtools}
\usepackage{enumerate}
\usepackage{dsfont}
\usepackage{mathptmx}
\usepackage[scaled=.90]{helvet}

\usepackage{setspace}

\usepackage{wrapfig}
\usepackage[margin=0.79in]{geometry}
\usepackage{leftidx}
\usepackage{float}
\usepackage[utf8]{inputenc}
\usepackage[hidelinks]{hyperref} 
\usepackage{caption}
\usepackage{import}
\usepackage{hyperref}

\usepackage{tikz, mathrsfs}
\usepackage{tikz-cd}
\usepackage{pgfplots}
\usetikzlibrary{calc}
\usetikzlibrary{intersections, pgfplots.fillbetween}

\usetikzlibrary{arrows}

\tikzcdset{scale cd/.style={every label/.append style={scale=#1},
    cells={nodes={scale=#1}}}}

\theoremstyle{definition}
\newtheorem{thm}{Theorem}[section]
\newtheorem{lemma}[thm]{Lemma}

\newtheorem{proposition}[thm]{Proposition}

\newtheorem{example}[thm]{Example}
\newtheorem{definition}[thm]{Definition}

\def\map{\mathop{\rm map}\nolimits}

\def\dirlim{\mbox{\,\hbox{lim}\kern-1.5em
                  \lower1.5ex\hbox{$\longrightarrow$}\,}}

\def\R{{\mathbb{R}}}

\def\S{\mathcal{S}}

\def\reeb{\mathcal{G}}
\def\truncreeb{\mathcal{T}}

\def\homology{\mathrm{H}}

\def\id{\mathrm{id}}

\def\pr{\mathrm{pr}}

\def\coker{\mathrm{coker}}
\def\Sect{\mathrm{Sect}}

\def\from{\colon}
\def\to{\rightarrow}

\DeclareMathAlphabet{\mathcal}{OMS}{cmsy}{m}{n}	

\DeclareRobustCommand{\coprod}{\mathop{\text{\fakecoprod}}}
\newcommand{\fakecoprod}{%
  \sbox0{$\prod$}%
  \smash{\raisebox{\dimexpr.9625\depth-\dp0}{\scalebox{1}[-1]{$\prod$}}}%
  \vphantom{$\prod$}%
}	

\title{Reeb complexes and topological persistence}

\author{Melvin Vaupel, Erik Hermansen, Paul Trygsland}

\date{}

\begin{document}

\blfootnote{\scriptsize{Email: melvin.vaupel@ntnu.no, erik.hermansen@ntnu.no, paul.trygsland@ntnu.no}}

\blfootnote{\scriptsize{Departement of Mathematical Sciences, Norwegian University of Science and Technology, Trondheim, Norway}}

\maketitle

\begin{abstract}
    \noindent We introduce Reeb complexes in order to capture how generators of homology flow along sections of a real valued continuous function. This intuition suggests a close relation of Reeb complexes to established methods in topological data analysis such as levelset zigzags and persistent homology. We make this relation precise and in particular explain how Reeb complexes and levelset zigzags can be extracted from the first pages of respective spectral sequences with the same termination.
\end{abstract}
\section{Introduction}
In this paper we study two different structures associated to a real-valued continuous function~$f \from X \to \R$. 
\begin{enumerate}
    \item Covers~$U=\{U_a = f^{-1}(I_a) \subset X \}$ of a topological space~$X$ pulled back from covers~$\{I_a \subset \R\}$ of the real line.
    \item Sections of~$f$, that is, for real numbers~$a \leq b$ continuous maps~$s \from [a,b] \to X$, such that~$f \circ s =\id$. 
\end{enumerate}
In both cases the available information is neatly organized in a \textit{simplicial space}. \\
For the pulled back cover~$U$, this is the well known \textit{\v{C}ech complex}.
 \begin{equation*}
    \begin{tikzcd}
    \coprod U_{\alpha_0} & \coprod U_{\alpha_0 \alpha_1} \arrow[l,yshift=0.5ex] \arrow[l,yshift=-0.5ex]  & \coprod U_{\alpha_0 \alpha_1 \alpha_2}  \arrow[l,yshift=0.7ex] \arrow[l] \arrow[l,yshift=-0.7ex] \cdots
    \end{tikzcd}
\end{equation*}
It combines the \textit{topology} of intersections of cover elements $U_{a_0 \ldots a_n}= U_{a_0}\cap \ldots \cap U_{a_n}$ with the \textit{combinatorics} of the various possible inclusions. Combining these two pieces of information, it is possible to recover the homology of the base space~$X$. Indeed, it can be shown that the realization of the \v{C}ech complex is homotopy equivalent to~$X$. To compute the homology of~$X$ from the \v{C}ech complex it can be helpful to utilize a spectral sequence. It turns out that various objects of interest to topological data analysis can be extracted from the first page of this spectral sequence, i.e. as intermediate steps of the corresponding homology computation. In Section \ref{section:cech}, we exemplify this with the persistence modules of a filtration \cite{carlsson2009topology} and the \textit{levelset zigzag} of a real valued height function like defined in \cite{carlsson2009zigzag}. \\

The simplicial space associated to the sections between a subset of heights~$A \subset \R$ is the \textit{section complex}. 
\begin{equation*}
   S_f^A \from \begin{tikzcd}
    (\S^A_f)_0 & (\S^A_f)_1 \arrow[l,yshift=0.5ex] \arrow[l,yshift=-0.5ex]  & (\S^A_f)_2  \arrow[l,yshift=0.7ex] \arrow[l] \arrow[l,yshift=-0.7ex] \dots
    \end{tikzcd}
\end{equation*} 
In analogy to the \v{C}ech complex, the section complex encodes the topology and combinatorics of sections. We discuss it thoroughly at the beginning of Section \ref{section:SectionComplex}, but in short: the topology is supposed to capture how sections relate via homotopies, while the combinatorics encode the various ways to concatenate sections. \\
Applying homology levelwise to the section complex yields 
\begin{equation*}
	\reeb_q^A \from \begin{tikzcd}
	 \homology_q (\S^A_f)_0 & \homology_q (\S^A_f)_1 \arrow[l,yshift=0.5ex] \arrow[l,yshift=-0.5ex]  & \homology_q(\S^A_f)_2  \arrow[l,yshift=0.7ex] \arrow[l] \arrow[l,yshift=-0.7ex] \dots  
	 \end{tikzcd}
 \end{equation*} 
 - a simplicial abelian group that we call the \textit{Reeb complex}. It is an object that captures how generators of homology flow between fibers of~$f$ along sections. 
 We may arrange the Reeb complexes as the first page of a spectral sequence, called the \textit{section spectral sequence}. In \cite{trygsland2021} it is shown how, under mild regularity assumptions on~$f$, this sequence collapses on the second page and computes the homology of the base space~$X$. The Reeb complexes can then be understood as an intermediate step in a calculation of the homology of~$X$. Thus, it comes at no surprise that they are intimately related to the well known persistence modules, that we can extract from the spectral sequence of the \v{C}ech complex. We make these relations precise in Proposition \ref{prop:diamondPrinciple} and Proposition \ref{proposition:persistence}. \\
 It is worthwhile to emphasize that if~$f$ is a piecewise linear function, the Reeb complexes are amenable for practical computations. Indeed, combining Proposition \ref{proposition:comparereeb} of this paper with Corollary 4.4 of \cite{Vaupel_SectionComplex_2022} implies that for piecewise linear functions we can compute the Reeb complexes in terms of the simplicial theory developed in the latter paper. \\
 
To summarise, this paper relates the theory of sections developed in \cite{trygsland2021} and \cite{Vaupel_SectionComplex_2022} to topological data analysis. To this end we define Reeb complexes and prove a close correspondence to known objects in topological data analysis like the persistence module of a filtration and levelset zigzag modules. The existence of this correspondence is motivated by the following discussion of the \v{C}ech complex.

\newpage 
 
\section{The \v{C}ech complex and topological data analysis}\label{section:cech}
 We explain how to obtain the persistence module of a filtration as well as the levelset zigzag module of a Morse-type function by applying homology levelwise to the \v{C}ech complex of an appropriate cover. Then we exhibit these modules as constituting the first page of a spectral sequence, that computes the homology of the covered space. This perspective on topological data analysis is very related to the sheaf-theoretic ideas explored for example in \cite{de2016categorified}, \cite{bubenik2014categorification}, \cite{curry2016discrete} and \cite{curry2014sheaves}. Consequently this section makes no claim to be original. The intention is rather to present well-known methods of topological data analysis in a way that makes their relation to the Reeb complexes of Section \ref{section:SectionComplex} plausible.   

Given a covering~$U=(U_\alpha)_{\alpha\in\Sigma}$ of a topological space $X$, consider the following diagram of continuous maps:
\begin{equation*}
    \begin{tikzcd}
    \coprod U_{\alpha_0} & \coprod U_{\alpha_0 \alpha_1} \arrow[l,yshift=0.5ex] \arrow[l,yshift=-0.5ex]  & \coprod U_{\alpha_0 \alpha_1 \alpha_2}  \arrow[l,yshift=0.7ex] \arrow[l] \arrow[l,yshift=-0.7ex] \cdots
    \end{tikzcd}
\end{equation*}
Here~$U_{\alpha_0,\ldots,\alpha_n}$ denotes the intersection~$U_{\alpha_0} \cap \ldots \cap U_{\alpha_n}$ and the arrows are the various ways of omitting indices and then applying inclusions. Interpreting these maps as face maps yields a \textit{simplicial space} $\breve{C}(U)$, which is also commonly referred to as the \textit{\v{C}ech complex} of the cover~$U$. \\

Consider a continuous real-valued function~$f \from X \to \R$. We can pull back a cover of~$\R$ along $f$ and thereby obtain a cover of the space~$X$. The corresponding simplicial vector space of this cover is then often a well-known object in topological data analysis, like we now wish to demonstrate with two examples.   
\subparagraph{Levelset zigzag} Let~$I_k \subset \R$ be a finite collection of open intervals such that~$U_k=f^{-1}I_k$ forms an open cover of~$X$ for which~$U_k\cap U_l\neq \emptyset$ only if ~$l=k\pm1$. Applying homology to the \v{C}ech complex of this pullback cover gives
\begin{equation*}
    \begin{tikzcd}
    \bigoplus\limits_{k=1}^{n} \homology_q(U_k) & \bigoplus\limits_{k=1}^{n-1} \homology_q(U_k \cap U_{k+1}) \arrow[l,yshift=0.5ex] \arrow[l,yshift=-0.5ex] 
    \end{tikzcd}
\end{equation*}
We can then wrap out the direct sums to obtain a zigzag module:
\begin{equation*}
\begin{tikzcd}
\homology_q(U_1) \leftarrow \homology_q(U_1 \cap U_{2}) \rightarrow \homology_q(U_2) \leftarrow \homology_q(U_2 \cap U_{3}) \rightarrow \cdots \leftarrow \homology_q(U_{n-1}\cap U_n) \rightarrow \homology_q(U_n). 
\end{tikzcd}
\end{equation*}
Note that this was possible because of the absence of higher simplices or else the diagram would have been much more complicated. \\
Furthermore, if~$f$ is of \textit{Morse type} with~$n-1$ critical values~$c_1, \ldots,c_{n-1}$, we may arrange a cover that is pulled back from~$k+1$ intervals~$I_1,\ldots,I_{k+1} \subset \R$, like above, but with the further requirement~$c_k$ lies in the intersection~$I_k \cap I_{k+1}$. Then the above zigzag module is isomorphic to
\footnotesize
\begin{equation*}
\begin{tikzcd}
\homology_q(f^{-1}(-\infty,c_1]) \leftarrow \homology_q(f^{-1}c_1) \rightarrow \homology_q(f^{-1}[c_1,c_2]) \leftarrow \homology_q(f^{-1}c_2) \rightarrow \cdots \leftarrow \homology_q(f^{-1}c_n) \rightarrow \homology_q(f^{-1}[c_n,\infty)). 
\end{tikzcd}
\end{equation*}
\normalsize
This is the \textit{levelset zigzag module} of $f$.\\

The persistence module associated to a \textit{filtration} can in a certain way be seen as a special case of this construction. 

\subparagraph{Persistent homology} Consider a filtration of topological spaces:
\begin{equation*}
    \mathbb{X} \from X_{i_0} \hookrightarrow X_{i_1} \hookrightarrow \ldots \hookrightarrow X_{i_{n-1}} \hookrightarrow X_{i_n}
\end{equation*}
We denote its \textit{mapping telescope} by~$C_{\mathbb{X}}$. It is obtained as the colimit of the following diagram:
\begin{center}
    \begin{tikzcd}[column sep=0.1in]
         & X_{i_0} \arrow[dl,"(\id \text{,}i_1)"'] \arrow[dr,"(\iota\text{,}i_1)"] &  &  &  & X_{i_{n-1}} \arrow[dl,"(\id \text{,}i_{n})"'] \arrow[dr,"(\iota \text{,}i_n)"] & \\
         X_{i_0} \times [i_0,i_1] & & X_{i_1} \times [i_1,i_2] & \cdots &  X_{i_{n-1}} \times [i_{n-1},i_n]  & & X_{i_n} 
    \end{tikzcd}
\end{center}
and we get an induced height function~$f_{\mathbb{X}} \from C_{\mathbb{X}} \to \R$. Along this function, pull back a cover of~$C_{\mathbb{X}}$ from open intervals in~$I_k \subset \R$ as in the construction of the levelset zigzag above, so that~$i_k \in I_k \cap I_{k+1}$. We then obtain the zigzag 
\begin{equation*}
\begin{tikzcd}
\homology_q(f_{\mathbb{X}}^{-1}i_0) \overset{\alpha_0}{\leftarrow} \homology_q(f_{\mathbb{X}}^{-1}[i_0,i_1]) \overset{\beta_1}{\rightarrow}  \cdots \overset{\alpha_{n-1}}{\leftarrow} \homology_q(f_{\mathbb{X}}^{-1}[i_{n-1},i_n]) \overset{\beta_{n}}\rightarrow \homology_q(f_{\mathbb{X}}^{-1}i_n),
\end{tikzcd}
\end{equation*}
where we observe that all~$\alpha_k$ are isomorphisms. We may thus write the persistence module
\begin{equation*}
\begin{tikzcd}
    \homology_q (f_{\mathbb{X}}^{-1}i_0) \arrow[r,"\beta_1\circ \alpha_0^{-1}"] & \homology_q(f_{\mathbb{X}}^{-1}i_1) \arrow[r,"\beta_2\circ \alpha_1^{-1}"] & \cdots \arrow[r,"\beta_n\circ \alpha_{n-1}^{-1}"] & \homology_q(f_{\mathbb{X}}^{-1}i_n).
\end{tikzcd}
    \end{equation*}
Compare this to the usual persistence module associated to the filtration
\begin{equation*}
\homology_q \mathbb{X}: \begin{tikzcd}
    \homology_q X_{i_0} \arrow[r] & \homology_q X_{i_1} \arrow[r] & \cdots \arrow[r] & \homology_q(X_{i_n}),
\end{tikzcd}
    \end{equation*}
and observe that they are isomorphic (0-interleaved).

\subparagraph{The spectral sequence of an open cover}
 We return to a general cover~$U=(U_\alpha)_{\alpha\in\Sigma}$ of $X$. Applying any homology functor levelwise to the \v{C}ech complex of this cover gives 
\begin{equation*}
    \begin{tikzcd}
    \bigoplus \homology_q(U_{\alpha_0}) & \bigoplus \homology_q(U_{\alpha_0 \alpha_1}) \arrow[l,yshift=0.5ex] \arrow[l,yshift=-0.5ex]  & \bigoplus \homology_q(U_{\alpha_0 \alpha_1 \alpha_2})  \arrow[l,yshift=0.7ex] \arrow[l] \arrow[l,yshift=-0.7ex] \cdots,
    \end{tikzcd}
\end{equation*}
which we recognize as a \textit{simplicial abelian group}. We may turn it into a chain complex  
 \begin{equation*}
    \begin{tikzcd}
    \bigoplus \homology_q(U_{\alpha_0}) & \bigoplus \homology_q(U_{\alpha_0 \alpha_1}) \arrow[l,"\partial^1_{1,q}"']  & \bigoplus \homology_q(U_{\alpha_0 \alpha_1 \alpha_2})  \arrow[l,"\partial^1_{2,q}"'] \cdots.
    \end{tikzcd}
\end{equation*}
The differential~$\partial^1_{p,q}$ is induced in~$\homology_q$ from the alternating sum of face maps. For example, for~$p=1$ these are~$\partial^1_{1,q}=\homology_q d_0-\homology_q d_1$. Organizing all the chain complexes for different~$q$'s gives us the first page of a \textit{spectral sequence}. 

\begin{center}
\begin{tikzpicture}
	\draw (0,0) -- (8.5,0);
	\draw (0,0) -- (0,4.5);
	
	\draw[->] (3.5,1.1) -- (2.35,1.1);
	\draw[->] (3.5,2.6) -- (2.35,2.6);
	\draw[->] (3.5,4.1) -- (2.35,4.1);
	
	\draw[->] (6.75,1.1) -- (5.6,1.1);
	\draw[->] (6.75,2.6) -- (5.6,2.6);
	\draw[->] (6.75,4.1) -- (5.6,4.1);
	
	\node [above] at (3,1) {$\partial^1_{1,0}$};
	\node [above] at (3,2.5) {$\partial^1_{1,1}$};
	\node [above] at (3,4) {$\partial^1_{1,2}$};
	
	\node [above] at (6.25,1) {$\partial^1_{2,0}$};
	\node [above] at (6.25,2.5) {$\partial^1_{2,1}$};
	\node [above] at (6.25,4) {$\partial^1_{2,2}$};
	
	\node [] at (1.5,1) {$\bigoplus\limits_{a_0\in \Sigma}\homology_0  U_{a_0}$};
	\node [] at (1.5,2.5) {$\bigoplus\limits_{a_0 \in \Sigma}\homology_1  U_{a_0}$};
	\node [] at (1.5,4) {$\bigoplus\limits_{a_0 \in \Sigma}\homology_2  U_{a_0}$};
	
	\node [] at (4.5,1) {$\bigoplus\limits_{a_0,a_1 \in \Sigma}\homology_0  U_{a_0a_1}$};
	\node [] at (4.5,2.5) {$\bigoplus\limits_{a_0,a_1 \in \Sigma}\homology_1  U_{a_0a_1}$};
	\node [] at (4.5,4) {$\bigoplus\limits_{a_0,a_1 \in \Sigma}\homology_2  U_{a_0a_1}$};
	
	\node [] at (8,1) {$\bigoplus\limits_{a_0,a_1,a_2 \in \Sigma}\homology_0  U_{a_0a_1a_2}$};
	\node [] at (8,2.5) {$\bigoplus\limits_{a_0,a_1,a_2 \in \Sigma}\homology_1  U_{a_0a_1a_2}$};
	\node [] at (8,4) {$\bigoplus\limits_{a_0,a_1,a_2 \in \Sigma}\homology_2  U_{a_0a_1a_2}$};

	\node [right] at (8.5,0) {$p$};
	\node [above] at (0,4.5) {$q$};
\end{tikzpicture}
\end{center}
Due to a result that goes back to Segal \cite{segal1968classifying}, which was later generalized in \cite{dugger2004topological}, we know that the termination of this sequence is always~$\homology_q X$ - the homology of the covered space. Furthermore, for  finite covers the sequence will eventually collapse and we are thus able to recover the homology of~$X$ by combining the \textit{topology} and \textit{combinatorics} of the cover~$U$. In practise this is achieved by solving Mayer-Vietoris-like extension problems to compute higher differentials. \\
Let us look at two cases in which these computations are particularly straightforward. The first is the one of a cover~$U=(U_{\alpha})_{\alpha\in\Sigma}$ of $X$ for which all intersections are contractible, also commonly referred to as a \textit{good cover}. Then, only the lowest row of the above first spectral sequence page is non-trivial. Furthermore, all the summands in its terms are either singletons or zero. We recognize it as the chain complex associated to the \textit{nerve} of the cover~$U$. In this case, the \textit{topology} of the cover is trivial and we can recover the homology of~$X$ just with the \textit{combinatorics} of the cover, encoded in its nerve. \\
As another special case, we consider a cover that has at most non-empty pairwise intersections. Then, just the first two columns are non-trivial. Again, the spectral sequence collapses on the second page and we read of the homology of~$X$ as 
\begin{align*}
\homology_q(X) \cong \begin{cases} \coker \partial_{1,0}^1 \quad q = 0 \\
                  \ker\partial_{1,q-1}^1 \oplus \coker\partial_{1,q}^1 \quad q \geq 1 
                  \end{cases}
\end{align*}

Note that in particular we can use this strategy to compute the homology of~$X$ from the levelset zigzag modules described above. \\

\newpage

\section{Reeb complexes}\label{section:SectionComplex}
\subparagraph{The section complex}
Let~$f \from X \to \mathbb{R}$ be a continuous function on a topological space. A~\emph{section} of~$f$ between heights~$a\leq b$ is a continuous map~$\rho \colon [a,b]\rightarrow X$, such that the composition~$f\circ \rho$ is the inclusion~$[a,b]\hookrightarrow \mathbb{R}$. These sections assemble into~$\Sect_f[a,b]$ - a subspace of the mapping space~$\map([a,b],X)$ with the compact-open topology. Fix a subset~$A \subset \R$. We define a series of spaces~$(\S^A_f)_0,(\S^A_f)_1,(\S_f^A)_2,\ldots$
\begin{itemize}
    \item The space~$(\S_f^A)_0$ is given as $\coprod\limits_{a \in A} f^{-1}(a)$, i.e. the disjoint union of fibers of the map~$f$.
    \item To obtain $(\S^A_f)_1$ collect all the sections going between heights in~$A$ into one \textit{space of sections} by taking the disjoint union over ordered pairs in~$A$, that is $\coprod\limits_{a\leq b} \Sect_f[a,b]$.
    \item For~$p \geq 2$, the space~$(S^A_f)_p$ has as points all the ways to concatenate~$p$ sections. Let for example~$\sigma \in \mathrm{Sect}_f[a,b]$ and~$\rho \in \mathrm{Sect}_f[b,c]$ be two sections with compatible ending and starting points. These can be concatenated to a section~$\sigma \ast \rho \in \mathrm{Sect}_f[a,c]$. We denote all possible ways to obtain such  concatenations 
    by~$\Sect_f[a,b,c]$. Then we induce the structure of a topological space from~$\Sect_f[a,b]$ and~$\Sect_f[b,c]$. Again, taking the disjoint union over all triples of heights~$a \leq b \leq c$, gives us the space~$(\S^A_f)_2$.
\end{itemize}

These spaces may be collected as in the following diagram.
\begin{equation*}
    \begin{tikzcd}
    (\S^A_f)_0 & (\S^A_f)_1 \arrow[l,yshift=0.5ex] \arrow[l,yshift=-0.5ex]  & (\S^A_f)_2  \arrow[l,yshift=0.7ex] \arrow[l] \arrow[l,yshift=-0.7ex] \dots
    \end{tikzcd}
\end{equation*}

The arrows denote the various ways to naturally map from~$(\S^A_f)_{p+1}$ to~$(\S^A_f)_p$. For example:~$\sigma \in \mathrm{Sect}_f[a,b]$ and~$\rho \in \mathrm{Sect}_f[b,c]$ as above associate to a point in~$\Sect_f[a,b,c]$ and thus in~$(\S^A_f)_2$. We can map this point to~$\sigma \in \Sect_f[a,b]$, $\rho \in \Sect_f[b,c]$ or~$\sigma \ast \rho \in \Sect_f[a,c]$, all of which lie in~$(\S^A_f)_1$. In this way we obtain three continuous maps~$(\S^A_f)_1 \leftarrow (\S^A_f)_2$. 
For~$p>2$, the~$p+1$ maps~$(\S^A_f)_{p} \leftarrow (\S^A_f)_{p+1}$ are obtained in the same manner. 
The two arrows~$(\S^A_f)_1$ to~$(\S^A_f)_0$ finally correspond to the two ways to evaluate a section at its end-points.

The above diagram defines a simplicial space. We denote it by~$\S^A_f$ and call it the \textit{section complex}. In the same way that the \v{C}ech complex~$\breve{C}(U)$ encodes how the topology of the cover~$U$ combinatorially fits together, the section complex~$\S^A_f$ contains the information how topological information about sections between heights in~$A$ combinatorially fits together.   

\subparagraph{Reeb complexes} In Section \ref{section:cech}, we applied homology functors to the \v{C}ech complex and obtained simplicial abelian groups, that we eventually turned into chain complexes. These chain complexes were then collected on the first page of a spectral sequence. In the previous paragraph we then reviewed the section complex~$\S_f^A$ that encoded information about sections of the function~$f$. We may now apply homology levelwise to the section complex as well. Intuitively, this gives an object that captures how homological information flows between fibers of~$f$ along sections.\\  

\begin{definition}
The~$q$'th Reeb complex associated to a continuous function~$f \from X \to \R$ and a subset~$A\subset \R$ is defined to be the simplicial vector space denoted by~$\reeb^A_q$, given as
\begin{equation*}
    \begin{tikzcd}
    \homology_q (\S^A_f)_0 & \homology_q (\S^A_f)_1 \arrow[l,yshift=0.5ex] \arrow[l,yshift=-0.5ex]  & \homology_q (\S^A_f)_2  \arrow[l,yshift=0.7ex] \arrow[l] \arrow[l,yshift=-0.7ex] \dots,
    \end{tikzcd}
\end{equation*}
that is, by applying the~$q$'th homology functor levelwise to~$\S^A_f$. 
\end{definition}
As for the \v{C}ech complexes, we may organize all the Reeb complexes as the rows of the first page of a spectral sequence by considering their corresponding chain complexes. 
\begin{center}
\begin{tikzpicture}
	\draw (0,0) -- (8.5,0);
	\draw (0,0) -- (0,4.5);
	
	\draw[->] (3.5,1.1) -- (2.25,1.1);
	\draw[->] (3.5,2.6) -- (2.25,2.6);
	\draw[->] (3.5,4.1) -- (2.25,4.1);
	
	\draw[->] (6.75,1.1) -- (5.5,1.1);
	\draw[->] (6.75,2.6) -- (5.5,2.6);
	\draw[->] (6.75,4.1) -- (5.5,4.1);
	
	\node [above] at (3,1) {$\partial^1_{1,0}$};
	\node [above] at (3,2.5) {$\partial^1_{1,1}$};
	\node [above] at (3,4) {$\partial^1_{1,2}$};
	
	\node [above] at (6.25,1) {$\partial^1_{2,0}$};
	\node [above] at (6.25,2.5) {$\partial^1_{2,1}$};
	\node [above] at (6.25,4) {$\partial^1_{2,2}$};
	
	\node [] at (1.5,1) {$\homology_0(\S^A_f)_0$};
	\node [] at (1.5,2.5) {$\homology_1(\S^A_f)_0$};
	\node [] at (1.5,4) {$\homology_2(\S^A_f)_0$};
	
	\node [] at (4.5,1) {$\homology_0(\S^A_f)_1$};
	\node [] at (4.5,2.5) {$\homology_1(\S^A_f)_1$};
	\node [] at (4.5,4) {$\homology_2(\S^A_f)_1$};
	
	\node [] at (8,1) {$\homology_0(\S^A_f)_2$};
	\node [] at (8,2.5) {$\homology_1(\S^A_f)_2$};
	\node [] at (8,4) {$\homology_2(\S^A_f)_2$};

	\node [right] at (8.5,0) {$p$};
	\node [above] at (0,4.5) {$q$};
\end{tikzpicture}
\end{center}

\subparagraph{Truncated Reeb complexes} We now assume that~$f$ is a Reeb function with finitely many critical height levels~$A=(c_1 < \ldots < c_n)$. The class of Reeb functions includes Morse functions on smooth manifolds and piecewise linear functions on CW-complexes. See Definition 2.6 of \cite{trygsland2021} for a precise definition. We then take the following truncation of the~$q$'th Reeb complex, just considering sections between adjacent critical levels
\begin{equation*}
    \begin{tikzcd}
    \bigoplus\limits_{i=1}^n \homology_q f^{-1}(c_i) & \bigoplus\limits_{i=1}^{n-1} \homology_q \Sect_f[c_i,c_{i+1}] \arrow[l,yshift=0.5ex] \arrow[l,yshift=-0.5ex].
    \end{tikzcd}
\end{equation*}
We denote this object as~$\truncreeb_q^f$. The complex~$\mathcal{T}_q^f$ is much smaller than the original Reeb complex~$\reeb_q^A$, which makes the following statement valuable for computations. 
\begin{proposition}\label{proposition:comparereeb}
For~$f\from X \to \R$ a Reeb function, let~$A$ be its set of critical values. Then the chain complexes associated to~$\reeb_q^A$ and~$\truncreeb_q^f$ are quasi-isomorphic.
\end{proposition}
\begin{proof}
Because~$A$ contains all the critical values of the Reeb function~$f$, we know from Proposition 4.2 of \cite{trygsland2021} that the spectral sequence associated to~$\reeb_q^A$ converges on the second page and that~$\homology_p \reeb^A_q = 0$ for~$p \geq 2$. Furthermore, we recognize the differential induced by the facemaps of~$\truncreeb_q^f$ as the \textit{critical differential}
\begin{equation*}
    \partial^s_{1,q} \from \bigoplus_{c_i} \homology_q \Sect_f[c_i,c_{i+1}] \to \bigoplus_{c_i} \homology_q f^{-1}(c_i)
\end{equation*}
as defined in Section 4.2 of \cite{trygsland2021}. Then, by Proposition 4.9 of \cite{trygsland2021},~$\homology_p\reeb^A_q \cong \homology_p \truncreeb^f_q$. 
\end{proof}

\subparagraph{Zigzag persistence and Reeb complexes} We can wrap out the direct sums of the complex~$\truncreeb_q^f$ into a zigzag module 
\begin{equation*}
    \begin{tikzcd}
\homology_q f^{-1}a_1 \leftarrow \homology_q\Sect_f[a_1,a_2] \rightarrow  \cdots \leftarrow \homology_q \Sect_f [a_{n-1},a_n] \rightarrow \homology_q f^{-1}(a_n). 
\end{tikzcd}
\end{equation*}
Comparing this zigzag to the levelset zigzag of~$f$, we notice two differences: 
\begin{enumerate}[1)]
    \item $\Sect_f[a_{i-1},a_{i}]$ and~$f^{-1}[a_{i-1},a_i]$ are different spaces in general and
    \item the arrows in the levelset zigzag are reversed compared to~$\truncreeb_q^f$.
\end{enumerate}
The following example illuminates these differences:

\begin{example}
Consider a cylinder with pinched boundary circles
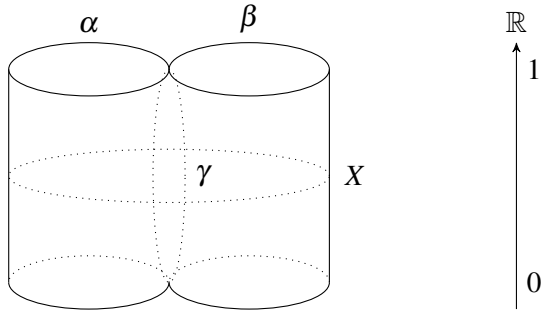
\begin{center}
\begin{tikzpicture}[scale= 0.7]
   \draw[->] (16,-1.5) -- (16,3.5);
    \node[right] at (16,-1) {$0$};
    \node[right] at (16,3) {$1$};
    \node[above] at (16,3.5) {$\mathbb{R}$};
    
    \draw[] (9.5,3) arc [start angle=0,end angle=360, x radius=1.5cm, y radius=0.5cm];
    \draw[] (12.5,3) arc [start angle=0,end angle=360, x radius=1.5cm, y radius=0.5cm];
    \draw[dotted] (9.5,-1) arc [start angle=0,end angle=180, x radius=1.5cm, y radius=0.5cm];
    \draw[] (6.5,-1) arc [start angle=-180,end angle=0, x radius=1.5cm, y radius=0.5cm];
    \draw[dotted] (12.5,-1) arc [start angle=0,end angle=180, x radius=1.5cm, y radius=0.5cm];
    \draw[] (9.5,-1) arc [start angle=-180,end angle=0, x radius=1.5cm, y radius=0.5cm];
    \draw[dotted] (12.5,1) arc [start angle=0,end angle=360, x radius=3cm, y radius=0.5cm];
    \draw[dotted] (9.8,1) arc [start angle=0,end angle=360, x radius=0.3cm, y radius=2cm];
    \draw[] (6.5,-1) -- (6.5,3);
    \draw[] (12.5,-1) -- (12.5,3);
    \node[right] at (12.6,1) {$X$};
    \node[above] at (8,3.5) {$\alpha$};
    \node[above] at (11,3.5) {$\beta$};
    \node[right] at (9.8,1) {$\gamma$};
    \end{tikzpicture}
\end{center}
with a mapping to $\R$ defined as 
\begin{equation*}
    \begin{tikzcd}
        S^1 \times [0,1] \arrow[r,"\pr_1"] \arrow[rr,bend right,"f"'] & \left[0,1\right] \arrow[r,hookrightarrow] & \R.
    \end{tikzcd}
\end{equation*}
We compute~$\truncreeb_0^f$ and~$\truncreeb_1^f$ in coordinates:

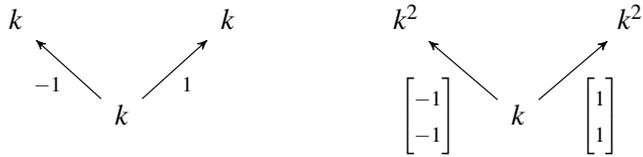
\begin{center}
    \begin{tikzcd}[ampersand replacement=\&,column sep=2.25em]
        k \& \& k \& \& k^2 \& \& k^2\\
         \& k\arrow[ul, "-1"] \arrow[ur, "1", swap]\& \& 
         \& \& k\arrow[ul, "
         \begin{bmatrix}
             -1\\
             -1
         \end{bmatrix}
         "] \arrow[ur,"
         \begin{bmatrix}
             1\\
             1
         \end{bmatrix}
         ", swap]\&\\
    \end{tikzcd}
\end{center}
The pre-image~$h^{-1}(0,1)=X$ deformation retracts onto the two horizontal circles~$\alpha$,~$\beta$ and the vertical circle~$\gamma$ depicted above. Pick these three circles as generators in~$\homology_1$ to calculate~$\homology_0$ and~$\homology_1$ of the corresponding levelset zigzags in coordinates:
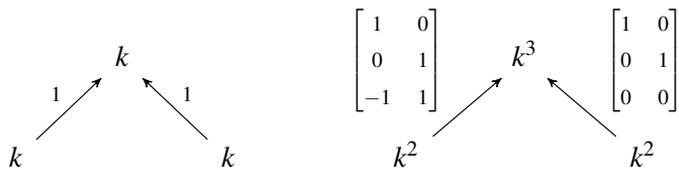
\begin{center}
    \begin{tikzcd}[ampersand replacement=\&,column sep=2.25em]

         \& k \& \&  \& \& k^3  \&\\
        
        k\arrow[ur, "1"] \& \& k\arrow[ul, "1", swap] \& \& k^2 \arrow[ur, "
         {\begin{bmatrix}
             1 & 0\\
             0 & 1\\
             -1 & 1
         \end{bmatrix}}
         ", near start]
         \& \& k^2\arrow[ul,"
         {\begin{bmatrix}
             1 & 0\\
             0 & 1\\
             0 & 0
         \end{bmatrix}}
         ", near start, swap]
    \end{tikzcd}
\end{center}
Note the distinct difference both in zeroth and first homology. However, taking direct sums across the middle rows in the concatenated diamonds
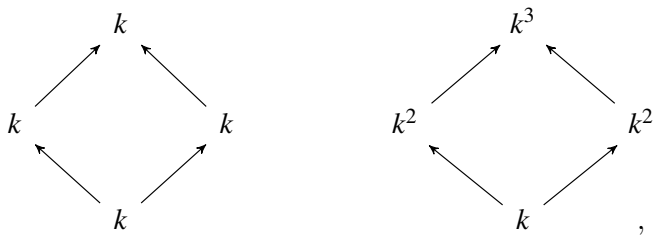
\begin{center}
    \begin{tikzcd}[ampersand replacement=\&,column sep=2.25em]
        \& k \& \& \& \& k^3\&\\
        k\arrow[ur] \& \& k\arrow[ul] \& \& k^2 \arrow[ur]\& \& k^2\arrow[ul]\\
        \& k \arrow[ul]\arrow[ur]\& \& \& \& k\arrow[ul]\arrow[ur]\&,
    \end{tikzcd}
\end{center}
results in sequences
\begin{center}
    \begin{tikzcd}[ampersand replacement=\&,column sep=5em]
        k \arrow[r, "{\begin{bmatrix}
         -1\\
         1
        \end{bmatrix}}"] \& k^2 \arrow[r, "{\begin{bmatrix}
         1 & 1
        \end{bmatrix}}"] \& k \&
        k \arrow[r, "{\begin{bmatrix}
         -1\\
         -1\\
         1\\
         1
        \end{bmatrix}}"] \& k^4 \arrow[r, "{\begin{bmatrix}
        1 & 0 & 1 & 0\\
        0 & 1 & 0 & 1\\
        -1 & 1 & 0 & 0
        \end{bmatrix}}"] \& k^3
    \end{tikzcd}.
\end{center}
that are exact in the middle term.
In this example, we can thus translate between the barcode of~$\truncreeb^f_q$ and the levelset zigzag modules via the diamond principle in~\cite{carlsson2009zigzag}.
\end{example}

The above observation generalises as we now demonstrate.
\begin{proposition}[Diamond Principle]
\label{prop:diamondPrinciple}
Let~$f \from X \to \R$ be a Reeb function. 
Then, for every pair of successive critical values~$a<b$, the sequence
\[
\homology_q \Sect_f[a,b]\rightarrow \homology_qf^{-1}a \oplus \homology_q f^{-1}b \rightarrow \homology_q f^{-1}[a,b]
\]
is exact at the middle term.
\end{proposition}
\begin{proof}
Evaluation at~$\frac{a+b}{2}$ defines a homotopy equivalence~$\mathrm{Sect}_f[a,b] \rightarrow f^{-1}(\frac{a+b}{2})$ by Proposition 3.10 in~\cite{trygsland2021}. The homotopy inverse is given by associating canonical sections (\textit{flow lines}) to points in the intermediate fiber~$f^{-1}(\frac{a+b}{2})$, which defines a map~$f^{-1}(\frac{a+b}{2})\rightarrow f^{-1}a \coprod f^{-1}b$. This gives a commutative ladder
\begin{center}
\begin{tikzcd}
    \homology_q \mathrm{Sect}_f [a,b]\arrow[r]\arrow[d] & \homology_q f^{-1}a \oplus \homology_q f^{-1}b \arrow[r]\arrow[d, equal]  & \homology_q f^{-1}[a,b]\arrow[d, equal] \\
    \homology_q f^{-1} (\frac{a+b}{2})\arrow[r] & \homology_q f^{-1}a \oplus \homology_q f^{-1}b \arrow[r] & \homology_q f^{-1}[a,b]
\end{tikzcd}
\end{center}
where all the vertical arrows are isomorphisms. Let~$I_a$ and~$I_b$ be open intervals in~$\mathbb{R}$ that contain~$a$ and~$b$, respectively. We can safely assume that~$a$ is the only critical value contained in~$f(I_a)$ and similarly that~$b$ is the only such value contained in~$f(I_b)$. Further we assume that the union~$f^{-1}I_a\cup f^{-1}I_b$ contains~$f^{-1}[a,b]$. By using the available inclusions, we get
\begin{center}
\begin{tikzcd}
    \homology_q f^{-1} (\frac{a+b}{2})\arrow[r]\arrow[d] & \homology_q f^{-1}a \oplus \homology_q f^{-1}b \arrow[r]\arrow[d] & \homology_q f^{-1}[a,b]\arrow[d]\\
    \homology_q f^{-1}I_a\cap f^{-1}I_b\arrow[r] & \homology_q f^{-1}I_a \oplus \homology_q f^{-1} I_b \arrow[r] & \homology_q f^{-1}I_a\cup f^{-1}I_b.
\end{tikzcd}
\end{center}
The vertical arrows are isomorphisms due to Lemma 2.8 in~\cite{trygsland2021}. We recognize the final row as part of the well-known Mayer-Vietoris sequence which is exact. 
\end{proof}

\subparagraph{Persistent homology and Reeb complexes} In Section \ref{section:cech}, we encountered the mapping telescope~$C_{\mathbb{X}}$ associated to a filtration 
\begin{equation*}
    \mathbb{X} \from X_{i_0} \hookrightarrow X_{i_1} \hookrightarrow \ldots \hookrightarrow X_{i_{n-1}} \hookrightarrow X_{i_n}.
\end{equation*}
This came with a height function~$f_{\mathbb{X}} \from C_{\mathbb{X}} \to \R$ of which we want to consider the section spaces. It turns out that in this case~$\truncreeb^f_q$ is isomorphic to the~$q$'th persistence module of the filtration. This will be a consequence of the following two lemmas. 
\begin{lemma}\label{lemma:persistencelemma1}
Let~$\iota \from X_{i_0} \hookrightarrow X_{i_1}$ be an inclusion, let~$C_{\iota}$ be the associated mapping cylinder and let~$f_{\iota} \from C_{\iota} \to \R$ be the induced height function that maps~$X_{i_0}$ to~$i_0$ and~$X_{i_1}$ to~$i_1$. Then the diagram 
\begin{equation*}
	\begin{tikzcd}
		& \Sect_{f_{\iota}}[i_0,i_1] \arrow[dl,"d_1"] \arrow[dr,"d_0"] & \\
		f_{\iota}^{-1}i_0 \arrow[rr,hookrightarrow] & & f_{\iota}^{-1}i_1
	\end{tikzcd},
\end{equation*} 
where~$d_0$ and~$d_1$ denote the respective face-maps in the section complex, commutes up to homotopy. 
\end{lemma}
\begin{proof}
	We construct a homotopy 
	\begin{equation*}
		\eta \from \Sect_{f_{\iota}}[i_0,i_1] \times [i_0,i_1] \to X_{i_1}
	\end{equation*} 
	It will be convenient to define it in terms of its adjoint
	\begin{equation*}
		\tilde{\eta} \from \Sect_{f_{\iota}}[i_0,i_1] \to \map([i_0,i_1],X_{i_1}). 
	\end{equation*}
	Postcomposing a section~$(\rho \from [i_0,i_1] \to C_{\iota}) \in \Sect_{f_{\iota}}[i_0,i_1]$ with the map~$\phi \from C_{\iota} \to X_{i_1}$ induced from the universal property of the pushout
	\begin{equation*}
		\begin{tikzcd}
			X_{i_0} \arrow[r] \arrow[d] & X_{i_1} \arrow[d] \arrow[ddr,"\id"] & \\
			X_{i_0} \times [i_0,i_1] \arrow[r] \arrow[rrd,"\iota \circ \pr_0"'] & C_{\iota} \arrow[dr,"\phi"] & \\
			& & X_{i_1}
		\end{tikzcd},
	\end{equation*}
	yields a continuous map~$\phi \circ \rho \from [i_0,i_1] \to X_{i_1}$. We define $\tilde{\eta}(\rho)=\phi \circ \rho$. 
\end{proof}
\begin{lemma}\label{lemma:persistencelemma2}
	In the setting of the previous lemma, the face map
	\begin{equation}
		d_1 \from \Sect_{f_{\iota}}[i_0,i_1]  \to X_{i_0}
	\end{equation} 
	is a homotopy equivalence. 
\end{lemma}
\begin{proof}
	This follows immediately from Proposition 4.10 in \cite{trygsland2021}. 
\end{proof}
We return to the section complex associated with the mapping telescope of the filtration~$\mathbb{X}$. The zigzag obtained from~$\truncreeb^f_q$ may be extended as follows: 
\begin{equation*}
	\begin{tikzcd}
		\homology_q \Sect_{f_{\mathbb{X}}}[i_0,i_1] \arrow[r] \arrow[d] & \homology_q X_{i_1} \arrow[d,equals] & \cdots \arrow[l] &  \homology_q \Sect_{f_{\mathbb{X}}}[i_{n-1},i_{n}] \arrow[d] \arrow[r] \arrow[l] & \homology_q X_{i_n} \arrow[d,equals] \\
		\homology_q X_{i_0} \arrow[r] & \homology_q X_{i_1} \arrow{r} &  \cdots \arrow[r] &  \homology_q X_{i_{n-1}} \arrow[r] & \homology_q X_{i_n} 
	\end{tikzcd}
\end{equation*}
where the lower row is just the ordinary~$q$'th persistence module of the filtration~$\mathbb{X}$, i.e. all the maps are induced by inclusions. We note that all squares in this diagram commute due to Lemma \ref{lemma:persistencelemma1}. Furthermore, all the arrows pointing to the left in the top row can be inverted due to Lemma \ref{lemma:persistencelemma2}. Thus,
\begin{proposition} \label{proposition:persistence}
The commutative ladder
\begin{equation*}
	\begin{tikzcd}
		\homology_q \Sect_{f_{\mathbb{X}}}[i_0,i_1] \arrow[d] \arrow[r]& \homology_q X_{i_1} \arrow[d,equals] \arrow[r]  & \cdots \arrow[r] &  \homology_q \Sect_{f_{\mathbb{X}}}[i_{n-1},i_{n}]  \arrow[d] \arrow[r] & \homology_q X_{i_n} \arrow[d,equals] \\
		\homology_q X_{i_0} \arrow[r] & \homology_q X_{i_1} \arrow{r} &  \cdots \arrow[r] & \homology_q X_{i_{n-1}} \arrow[r] & \homology_q X_{i_n} 
	\end{tikzcd}
\end{equation*}
defines an isomorphism of persistence modules.
\end{proposition}

\newpage

\section*{Acknowledgements}
We would like to thank our supervisors Benjamin Dunn and Markus Szymik for their valuable input and the kind and motivating encouragement. This work was financially supported from the Department of Mathematical Sciences at the NTNU and from NTNU's Enabling Technologies Biotechnology program.

\addcontentsline{toc}{section}{References}
\bibliographystyle{amsalpha}
\bibliography{SectionSpacesTDA}

\end{document}